\documentclass[11pt]{amsart}

\newif\ifdraft\draftfalse

\usepackage{amssymb}
\usepackage{amsthm}
\usepackage{eucal}
\usepackage[english]{babel}
\usepackage{nicefrac}
\usepackage{times}
\usepackage{latexsym,amscd,amsmath,amsthm,amssymb}

\makeatletter
\def\@begintheorem#1#2[#3]{%
    \def\naam{#1}
  \deferred@thm@head{\the\thm@headfont \thm@indent
    \@ifempty{#1}{\let\thmname\@gobble}{\let\thmname\@iden}%
    \@ifempty{#2}{\let\thmnumber\@gobble}{\let\thmnumber\@iden}%
    \@ifempty{#3}{\let\thmnote\@gobble}{\let\thmnote\@iden}%
    \thm@swap\swappedhead\thmhead{#1}{#2}{#3}%
    \the\thm@headpunct
    \thmheadnl % possibly a newline.
    \hskip\thm@headsep
  }%
  \ignorespaces}
\makeatother

\newcommand{\kantlijndraft}[1]{\ifdraft\hspace{-\lastskip}%
\vadjust{\vspace{-1mm}\smash{\llap{{\tt#1}\hspace{8mm}}}\vspace{1mm}}\fi}

\def\voegToe#1#2#3{\immediate\write1{\string\newlabel{#1}{{#2}{#3}}}}

\newcommand{\thlabel}[1]{\voegToe{#1}{\naam\noexpand~\thetheorem-
}{\thepage}\kantlijndraft{#1}}

\makeatletter
\renewcommand{\label}[1]{\voegToe{#1}{\@currentlabel}{\thepage}\kantlijndraft{#1}}
\makeatother

\newtheorem{theorem}{Theorem}[section]

\newtheorem{corollary}[theorem]{Corollary}
\newtheorem{question}[theorem]{Question}

\newtheorem{proposition}[theorem]{Proposition}

\theoremstyle{definition}
\newtheorem{example}[theorem]{Example}
\newtheorem{definition}[theorem]{Definition}
\theoremstyle{remark}

\numberwithin{equation}{section}

\newtheorem{claim2}{\sc Claim}



\newcommand{\sse}{\subseteq}                           %subset or equal [\sse]

\newcommand{\minus}{\backslash}
\newcommand{\Un}{\bigcup}
\newcommand{\un}{\cup}
\newcommand{\Meet}{\bigcap}
\newcommand{\meet}{\cap}
\newcommand{\es}{\varnothing}                          %emptyset [\es]

\newcommand{\closure}[1]{\ensuremath{\overline{#1}}}

\newcommand{\scr}[1]{\ensuremath{\mathcal{#1}}}

\def\cprime{$'$}
\def\cont{\mathfrak{c}}
\def\sapirovskii{{\v{S}}apirovski{\u\i}}
\def\arhangelskii{Arhangel{\cprime}ski{\u\i}}

\def\lindelof{Lindel\"of}
\def\juhasz{Juh{\'a}sz}

\begin{document}

\title{On weakening tightness to weak tightness}

%    Information for authors:
\author{A. Bella}\address{Department of Mathematics, University
of Catania, Viale A. Doria 6, 95125 Catania, Italy}
\email{bella@dmi.unict.it}
\author{N. Carlson}\address{Department of Mathematics, California
Lutheran University, 60 W. Olsen Rd, MC 3750, Thousand Oaks, CA
91360 USA}
\email{ncarlson@callutheran.edu}

%%%% change the following:


\begin{abstract}
The weak tightness $wt(X)$ of a space $X$ was introduced~in
\cite{Car2018} with the property $wt(X)\leq t(X)$. We investigate
several well-known results concerning $t(X)$ and consider whether
they extend to the weak tightness setting. First we give an
example of a non-sequential compactum $X$ such that
$wt(X)=\aleph_0<t(X)$ under  $2^{\aleph_0}=2^{\aleph_1}$. In
particular,  this demonstrates the celebrated Balogh's
Theorem~\cite{Bal1989} does not hold in general if countably
tight is replaced with weakly countably tight. Second, we
introduce the notion of an S-free sequence and show that if $X$
is a homogeneous compactum then $|X|\leq 2^{wt(X)\pi_\chi(X)}$.
This refines a theorem of De la Vega~\cite{DeLaVega2006}. In the
case where the cardinal invariants involved are countable, this
also represents a variation of a theorem of~\juhasz~and van
Mill~\cite{JVM2018}. Third, we show that if $X$ is a $T_1$ space,
$wt(X)\leq\kappa$, $X$ is $\kappa^+$-compact, and
$\psi(\closure{D},X)\leq 2^\kappa$ for any $D\sse X$ satisfying
$|D|\leq 2^\kappa$, then a) $d(X)\leq 2^\kappa$ and b) $X$ has at
most $2^\kappa$-many $G_\kappa$-points. This is a variation of
another theorem of Balogh~\cite{Bal2003}. Finally, we show that
if $X$ is a regular space, $\kappa=L(X)wt(X)$, and $\lambda$ is a
caliber of $X$ satisfying $\kappa<\lambda\leq
\left(2^{\kappa}\right)^+$, then $d(X)\leq 2^{\kappa}$. This
extends of theorem of~\arhangelskii~\cite{Arh2000}.
\end{abstract}

\maketitle

%%%%%%%%%%%%%%%%%%%% section  %%%%%%%%%%%%%%%%%%%%%%%%%
\section{Introduction.}
The cardinal function $t(X)$, the \emph{tightness} of a
topological space $X$, is the least infinite cardinal $\kappa$
such that whenever $A\sse X$ and $x\in\closure{A}$ then there
exists $B\sse A$ such that $|B|\leq\kappa$ and $x\in\closure{B}$.
Motivated by results of~\juhasz~and van Mill in~\cite{JVM2018},
the cardinal function $wt(X)$, the \emph{weak tightness of }$X$,
was introduced in~\cite{Car2018}. (See
Definition~\ref{weakTightness} below). In~\cite{Car2018} it was
shown that $|X|\leq 2^{L(X)wt(X)\psi(X)}$ for any Hausdorff space
$X$. As $wt(X)\leq t(X)$ for any space $X$, this improved the
\arhangelskii-\sapirovskii~cardinality bound for Hausdorff
spaces~\cite{Arh1969},~\cite{Sap1972}.

In this study we explore other fundamental cardinal function
results involving tightness and consider whether $t(X)$ can be
replaced with $wt(X)$. In 1989 Balogh~\cite{Bal1989} answered the
Moore-Mr\'owca Problem by showing that every countably tight
compactum is sequential under the Proper Forcing Axiom (PFA).
In~\ref{example}, we give an example demonstrating that countably
tight cannot be replaced with weakly countably tight in Balogh's
Theorem. It is also a consistent  example of a compact group for
which $wt(X)=\aleph_0<t(X)$.

In ~\cite{DeLaVega2006}, R. de la Vega answered a long-standing
question of~\arhangelskii~by showing that $|X|\leq 2^{t(X)}$ for
a homogeneous compactum $X$. Recall that a space $X$ is
\emph{homogeneous} if for all $x,y\in X$ there exists a
homeomorphism $h:X\to Y$ such that $h(x)=y$. Roughly, a space is
homogeneous if all points in the space share identical
topological properties. In~\cite{JVM2018}, \juhasz~and van Mill
introduced new techniques and gave a certain improvement of De la
Vega's Theorem in the case where $X$ is a countable union of
countably tight subspaces. In Theorem~\ref{cpthomog}, we prove
that $|X|\leq 2^{wt(X)\pi_\chi(X)}$ for any homogeneous
compactum. In the case where the cardinal invariants involved are
countable, Theorem~\ref{cpthomog} represents a variation of
Theorem 4.1 in \cite{JVM2018}. As $wt(X)\leq t(X)$ for any space
and $\pi_\chi(X)\leq t(X)$ for any compact space, our result
gives a general improvement of De la Vega's Theorem. To this end
we introduce the notion of an S-free sequence in
Definition~\ref{sfree}. In compact spaces these sequences play a
role for $wt(X)$ similar to the role free sequences play for
$t(X)$.

S-free sequences are also used to replace $t(X)$ with $wt(X)$ in
a second theorem of Balogh~\cite{Bal2003}. In
Theorem~\ref{Balogh} a closing-off argument is used to show that
if $X$ is $T_1$, $\kappa$ is a cardinal, $wt(X)\leq\kappa$, $X$
is $\kappa^+$-compact, and $\psi(\closure{D},X)\leq 2^\kappa$ for
any $D\sse X$ satisfying $|D|\leq 2^\kappa$, then $d(X)\leq
2^\kappa$ and there are at most $2^\kappa$-many
$G_\kappa$-points. This result is related to Lemma 3.2
in~\cite{JVM2018} and produces an alternative proof that $|X|\leq
2^{L(X)wt(X)\psi(X)}$ for a Hausdorff space $X$.

In~\cite{Arh2000},~\arhangelskii~considered the notion of a
caliber and showed that if $X$ is a Lindel\"of, regular space,
$t(X)=\aleph_0$, and $\aleph_1$ is a caliber of $X$, then
$d(X)\leq\cont$. Recall a cardinal $\kappa$ is a \emph{caliber}
of a space X if every family of open sets of cardinality $\kappa$
has a subfamily of cardinality $\kappa$ with non-empty
intersection. We show in Theorem~\ref{caliber} that in this
result $t(X)=\aleph_0$ can be replaced with $wt(X)=\aleph_0$. It
also extends to the uncountable case where the cardinal
$\kappa=L(X)wt(X)$ is used.

Finally, we mention that  a rather easy modification of  the
argument used in Theorem 2.8 of \cite{Car2018}  proves $|X|\le
2^{aL_c(X)wt(X)\psi_c(X)}$ for every Hausdorff  space $X$.  This
generalizes to the weak tightness setting the analogous
inequality
established in \cite{belcam}.

By \emph{compactum} we mean a compact, Hausdorff space. For all
undefined notions we refer the reader to~\cite{Engelking}
and~\cite{Juhasz}.

%%%%%%%%%%%%%%%%%%%%%%%%%%%%%%%%%%%%%%%%%%%%%%%%%%%%%%%%%%%
\section{An example and Balogh's Theorem}
The weak tightness $wt(X)$~\cite{Car2018} of a space $X$ is
defined as follows.

\begin{definition}\label{weakTightness}
Let $X$ be a space. Given a cardinal $\kappa$ and $A\sse X$, the
$\kappa$-\emph{closure} of $A$ is defined as $cl_\kappa
A=\Un_{B\in[A]^{\leq\kappa}}\closure{B}$. The \emph{weak
tightness} $wt(X)$ of $X$ is the least infinite cardinal $\kappa$
for which there is a cover $\scr{C}$ of $X$ such that
$|\scr{C}|\leq 2^\kappa$ and for all $C\in\scr{C}$, a)
$t(C)\leq\kappa$ and b) $X=cl_{2^\kappa}C$. We say that $X$ is
\emph{weakly countably tight} if $wt(X)=\aleph_0$.
\end{definition}

The condition b) above can be difficult to work with. Instead a
convenient tool for demonstrating that the weak tightness of a
space $X$ is at most $\kappa$ was given in~\cite{Car2018}. This
replaces b) with the more tractable condition that $C$ is dense
in $X$ under mild assumptions involving $t(X)$ or
$\pi_\chi(X)$.

\begin{proposition}[\cite{Car2018}, Lemma 2.10] \label{wteasy}
Let $X$ be a space, $\kappa$ a cardinal, and $\scr{C}$ a cover of
$X$ such that $|\scr{C}|\leq 2^\kappa$, and for all
$C\in\scr{C}$, $t(C)\leq\kappa$ and $C$ is dense in $X$. If
$t(X)\leq 2^\kappa$ or $\pi_\chi(X)\leq 2^\kappa$ then
$wt(X)\leq\kappa$.
\end{proposition}
A $\sigma$-compact space $X$ for which $wt(X)<t(X)$ is provided
in \cite{Car2018}. In this section we will see that at least
consistently there is a compact example.

In 1989 Balogh~\cite{Bal1989} answered the Moore-Mr\'owca Problem
by proving under the Proper Forcing Axiom (PFA) that every
compactum of countable tightness is sequential. Since PFA implies
$2^{\aleph_0}=2^{\aleph_1}$,
 the following example demonstrates that in that theorem
countable tightness cannot be replaced with weak countable
tightness. It is also, more generally, a consistent  example of a
compact group $X$ such that $wt(X)=\aleph_0<t(X)$.
\begin{example}\label{example}
Assume  $2^{\aleph_0}=2^{\aleph_1}$  and let $X$ be the Cantor
cube $2^{\omega_1}$. Recall that $X$ has tightness
$t(X)=\aleph_1$ and is therefore not sequential. For each $x\in
X$, let $C_x=\{y\in X:|\{\alpha<\omega_1:x(\alpha)\neq
y(\alpha)\}|<\omega\}$. Observe for each $x\in X$ that $C_x$ is a
countably tight, dense subspace of $X$ and that
$\scr{C}=\{C_x:x\in X\}$ is a cover of $X$. Furthermore, our set-
theoretic assumption implies  $|\scr{C}|\leq 2^\omega$. As
$t(X)=\aleph_1\leq 2^\omega$, by Proposition~\ref{wteasy} it
follows that $wt(X)=\aleph_0$. Thus $X$ is a non-sequential
compact group of tightness $\aleph_1$ such that, under
$2^{\aleph_0}=2^{\aleph_1}$, $X$ is weakly countably tight. Note
also that $X$ is homogeneous.
\end{example}
As $\pi_\chi(X)=\aleph_1$, this example also shows that  \v
Sapirovski\u\i's   theorem
$\pi\chi(X)\le t(X)$ for any compactum $X$ may fail by using
the weak tightness, answering Question 4.8 in~\cite{Car2018}.

\juhasz~and van Mill~\cite{JVM2018} asked if there is a homogeneous compactum
that is the countable union of countably tight subspaces
($\sigma$-CT) yet not countably tight. They observe that such a
space can only exist in a model in which $2^{\aleph_0}=2^{\aleph_1}$. We
note that the above example is not $\sigma$-CT. (For if it were,
by Lemma 2.4 in~\cite{JVM2018} $X$ would have a point of
countable $\pi$-character). However, under $2^{\aleph_0}=2^{\aleph_1}$,
$X$ is nevertheless a $2^{\aleph_0}$-union of countably tight dense
subspaces.

In light of our example, we ask the following.

\begin{question}
Does there exist a compactum $X$ with $wt(X)<t(X)$ in ZFC?
\end{question}

\begin{question}
Does there exists a compactum that is not countably tight and is the countable union of dense countably tight subspaces?
\end{question}

%%%%%%%%%%%%%%%%%%%%%%%%%%%%%%%%%%%%%%%%%%%%%%%%%%%%%%%%%%%
\section{S-free sequences and homogeneous compacta}
The goal in this section is to give an improved bound for the
cardinality of any homogeneous compactum. For a space $X$, the
notion of a $\scr{C}$-saturated subset of a
space $X$ was introduced in \cite{JVM2018}. Given a cover
$\scr{C}$ of $X$, a subset $A\sse X$ is
$\scr{C}$-\emph{saturated} if $A\meet C$ is dense in $A$ for
every $C\in\scr{C}$. It is clear that the union of
$\scr{C}$-saturated subsets is $\scr{C}$-saturated.

Let $X$ be a space. If $wt(X)\leq\kappa$ then there is a cover
$\scr{C}$ of $X$ witnessing the properties
in~\ref{weakTightness}. By Lemma 3.2 in \cite{JVM2018}, for all
$x\in X$ we can fix a $\scr{C}$-saturated set $S(x)$ such that
$x\in S(x)$ and $|S(x)|\leq 2^\kappa$. For all $A\sse X$, set
$S(A)=\Un\{S(x):x\in A\}$ and note that $S(A)$ is
$\scr{C}$-saturated. We will use the existence of the cover
$\scr{C}$ and the family $\{S(A):A\sse X\}$ implicitly in proofs
where $wt(X)\leq\kappa$ without direct mention.

Recall that for a cardinal $\kappa$ and a space $X$, a set
$\{x_\alpha:\alpha<\kappa\}\sse X$ is a \emph{free sequence} if
$\closure{\{x_\beta:\beta<\alpha\}}\meet\closure{\{x_\beta:\alpha\leq\beta<\kappa\}}=\es$ for all $\alpha<\kappa$. We define
the notion of an S-free sequence below. Note that every S-free
sequence is a free sequence and that an S-free sequence is
defined  by a cover witnessing $wt(X)$.

\begin{definition}\label{sfree}
Let $wt(X)=\kappa $. A set
$\{x_\alpha:\alpha<\lambda \}$ is an \emph{S-free sequence} if
$\closure{S(\{x_\beta:\beta<\alpha\})}\meet\closure{\{x_\beta:\alpha\leq\beta<\lambda \}}=\es$ for all $\alpha<\lambda $.
\end{definition}

For a compactum $X$, it is well known that $t(X)=F(X)$, where
$F(X)=\sup\{|E|: E\textup{ is a free sequence in }X\}$. The
following proposition demonstrates that S-free sequences play a
similar role for $wt(X)$ in $\kappa^+$-compact spaces. Recall
that a space $X$ is $\kappa$-compact if every subset of
cardinality $\kappa$ has a complete accumulation point.

\begin{proposition}\label{Sfree}
If $X$ is $\kappa^+$-compact where $\kappa=wt(X)$ then $X$ does
not contain an S-free sequence of length $\kappa^+$.
\end{proposition}
\begin{proof}
Assume the contrary and let $A=\{x_\alpha:\alpha<\kappa^+\}$ be
an S-free sequence. Since $X$ is $\kappa^+$-compact, the set $A$
has a complete accumulation point $x$. Note that
$x\in\closure{S(A)}$. Now, the assumption $wt(X)=\kappa$ and the fact that $\closure{S(A)}=\Un_{\alpha<\kappa^+}S(\{x_\beta:\beta<\alpha\})$
(\cite{Car2018}, Lemma 2.7) ensure the existence of an ordinal
$\alpha<\kappa^+$ such that
$x\in\closure{S(\{x_\beta:\beta<\alpha\})}$. But this implies
that the set
$X\minus\closure{\{x_\beta:\alpha\leq\beta<\kappa^+\}}$ is a
neighborhood of $x$ meeting $A$ in less than $\kappa^+$-many
points, a contradiction.
\end{proof}

In Theorem 2.5 in \cite{JVM2018},~\juhasz~and van Mill showed
that if a compact space $X$ is a countable union of countable
tight subspaces then there is a non-empty $G_\delta$-set
contained in the closure of a countable set. That is, a
$\sigma$-CT compact space has a non-empty subseparable
$G_\delta$-set. The following theorem is a variation of this and,
in the countable case, states that if there is a cover $\scr{C}$
of $X$ such that $|\scr{C}|\leq\cont$, and for all $C\in\scr{C}$
a) $C$ is countably tight and b) $X=cl_\cont{C}$, then there is a
$G_\delta$-set contained in the closure of a set of size $\cont$.
The proof is a variation of the proof of 2.2.4 in \cite{Arh1978}
and uses S-free sequences. Recall that a set $G$ in compactum $X$
is a closed $G_\kappa$-set if and only if $\chi(G,X)\leq\kappa$.

\begin{theorem}\label{gh}
Let $X$ be a compactum and let $\kappa=wt(X)$. Then there exists
a non-empty closed set $G\sse X$ and $H\in[X]^{\leq 2^\kappa}$
such that $G\sse\closure{H}$ and $\chi(G,X)\leq\kappa$.
\end{theorem}

\begin{proof}
Define $\scr{F}=\{F\sse X: F\neq\es, F\textup{ is closed, and
}\chi(F,X)\leq\kappa\}$. Suppose by way of contradiction that for
all $F\in\scr{F}$ and all $H\in [X]^{\leq 2^\kappa}$ we have
$F\minus\closure{H}\neq\es$. We construct a decreasing sequence
$\scr{F}^\prime=\{F_\alpha: \alpha<\kappa^+\}\sse\scr{F}$ and an
S-free sequence $A=\{x_\alpha: \alpha<\kappa^+\}$.

For $\alpha<\kappa^+$, suppose $x_\beta\in X$ and $F_\beta$ have
been defined for all $\beta<\alpha<\kappa^+$, where
$F_{\beta^{\prime\prime}}\sse F_{\beta^\prime}$ for all
$\beta^\prime<\beta^{\prime\prime}<\alpha$. Define
$A_\alpha=\{x_\beta:\beta<\alpha\}$ and
$G_\alpha=\Meet\{F_\beta:\beta<\alpha\}$. As $X$ is compact,
$G_\alpha\neq\es$ and $G_\alpha\in\scr{F}$. As
$|A_\alpha|\leq\kappa$, we have $|S(A_\alpha)|\leq\kappa\cdot
2^\kappa=2^\kappa$. Therefore,
$G_\alpha\minus\closure{S(A_\alpha)}\neq\es$.

Let $x_\alpha\in G_\alpha\minus\closure{S(A_\alpha)}$. As
$\chi(G_\alpha,X)\leq\kappa$, we have that $G_\alpha$ is a
$G_\kappa$-set. There exists an open family $\scr{U}$ such that
$|\scr{U}|\leq\kappa$ and $G_\alpha=\Meet\scr{U}$. As $X$ is
regular, for all $U\in\scr{U}$ there exists a closed
$G_\delta$-set $G_U$ such that $x_\alpha\in G_U\sse
U\minus\closure{S(A_\alpha)}$. Set
$F_\alpha=\Meet\{G_U:U\in\scr{U}\}$, and note $x_\alpha\in
F_\alpha\sse G_\alpha\minus\closure{S(A_\alpha)}$ and that
$F_\alpha$ is a non-empty closed $G_\kappa$-set. As $X$ is
compactum, $\chi(F_\alpha,X)\leq\kappa$ and $F_\alpha\in\scr{F}$.
The choice of $x_\alpha$ and $F_\alpha$ completes the
construction of the sequences $\scr{F}^\prime$ and $A$.

We show now that $A$ is an S-free sequence. Again, let
$\alpha<\kappa^+$. Note $\{x_\beta:\alpha\leq\beta<\kappa^+\}\sse
F_\alpha$ as $\scr{F}^\prime$ is a decreasing sequence. Thus
$\closure{\{x_\beta:\alpha\leq\beta<\kappa^+\}}\sse F_\alpha$.
Also, $\{x_\beta:\beta<\alpha\}=A_\alpha$ and so
$\closure{S(\{x_\beta:\beta<\alpha\})}=\closure{S(A_\alpha)}\sse
X\minus F_\alpha$. Therefore,
$\closure{S(\{x_\beta:\beta<\alpha\})}\meet\closure{\{x_\beta:\alpha\leq\beta<\kappa^+\}}=\es$ and $A$ is S-free.
However, $X$ is $\kappa^+$-compact and cannot have an S-free
sequence of length $\kappa^+$ by Proposition~\ref{Sfree}. This is
a contradiction and thus there exists $G$ and $H$ as required.
\end{proof}

The following improvement of a theorem of Pytkeev~\cite{Pyt1985}
was given in~\cite{Car2018}. It is used in the proof of Theorem~\ref{cpthomogd}. For a cardinal $\kappa$, the space
$X_\kappa$ is the set $X$ with the collection of $G_\kappa$-sets
as a basis for its topology.

\begin{theorem}[\cite{Car2018}, Corollary 3.5]\label{wtPytkeev}
If $X$ is a compactum and $\kappa$ is a cardinal, then
$L(X_\kappa)\leq 2^{wt(X)\cdot\kappa}$.
\end{theorem}

Homogeneity is now applied in the next theorem.

\begin{theorem}\label{cpthomogd}
If $X$ is a homogeneous compactum then $d(X)\leq 2^{wt(X)}$.
\end{theorem}

\begin{proof}
Let $\kappa=wt(X)$. By Theorem~\ref{gh} there exists a non-empty
closed set $G\sse X$ and $H\in[X]^{\leq 2^\kappa}$ such that
$G\sse\closure{H}$ and $\chi(G,X)\leq\kappa$. Fix $p\in G$. As
$X$ is homogeneous for all $x\in X$ there exists a homeomorphism
$h_x:X\to X$ such that $h_x(p)=x$. Then $\scr{F}=\{h_x[G]:x\in
X\}$ is a cover of $X$ by compact sets of character at most
$\kappa$, and for all $x\in X$, $h_x[G]\sse\closure{h_x[H]}$ and
$|h_x[H]|\leq 2^\kappa$. By Theorem~\ref{wtPytkeev} there exists
$\scr{G}\in[\scr{F}]^{\leq 2^\kappa}$ and a family
$\{H_G:G\in\scr{G}\}$ such that $X=\Un\scr{G}$, and for all
$G\in\scr{G}$, $G\sse\closure{H_G}$ and $|H_G|\leq 2^\kappa$.

Let $D=\Un\{H_G:G\in\scr{G}\}$ and note that $|D|\leq
2^\kappa\cdot 2^\kappa=2^\kappa$. Observe that
$$X=\Un\scr{G}\sse\Un\{\closure{H_G}:G\in\scr{G}\}\sse\closure{D},$$
showing $D$ is dense in $X$. Therefore $d(X)\leq|D|\leq
2^\kappa$.
\end{proof}

Homogeneity is used in a different way in the next corollary.
Simply apply the above and the well-known fact that $|X|\leq
d(X)^{\pi_\chi(X)}$ for any homogeneous Hausdorff space. As
$wt(X)\leq t(X)$ for any space $X$, and $\pi_\chi(X)\leq t(X)$
for any compactum, Corollary~\ref{cpthomog} represents an
improvement of De la Vega's Theorem~\cite{DeLaVega2006}.

\begin{corollary}\label{cpthomog}
If $X$ is a homogeneous compactum then $|X|\leq
2^{wt(X)\pi_\chi(X)}$.
\end{corollary}
As Example~\ref{example} shows that $\pi_\chi(X)\leq wt(X)$ may fail for homogeneous compacta, we ask the following question. It was asked in~\cite{Car2018} for $wt(X)=\aleph_0$.

\begin{question}
If $X$ is a homogeneous compactum, is $|X|\leq 2^{wt(X)}$?
\end{question}

We also recall the following still-open question of R. de la Vega.

\begin{question}
If $X$ is a homogeneous compactum, is $|X|\leq 2^{\pi_\chi(X)}$?
\end{question}

Using Lemma 4.4 in~\cite{Car2018}, which is an adjustment of
Corollary 2.9 in~\cite{AVR2007}, and the fact that $|X|\leq
d(X)^{\pi_\chi(X)}$ for any power homogeneous Hausdorff
space~\cite{Rid2006}, the proofs of Theorem~\ref{cpthomogd} and
Corollary~\ref{cpthomog} can easily be generalized. Recall a
space $X$ is \emph{power homogeneous} if there exists a cardinal
$\kappa$ such that $X^\kappa$ is homogeneous. Power homogeneity
is thus a weaker notion of homogeneity.

\begin{theorem}
If $X$ is a power homogeneous compactum then $d(X)\leq
2^{wt(X)}$.
\end{theorem}

\begin{corollary}
If $X$ is a power homogeneous compactum then $|X|\leq
2^{wt(X)\pi_\chi(X)}$.
\end{corollary}

Below we isolate the case of the above corollary where all
cardinal invariants involved are countable. It follows directly
from Proposition~\ref{wteasy} and the above.

\begin{corollary}\label{phcountable}
Let $X$ be a power homogeneous compactum of countable
$\pi$-character with a cover $\scr{C}$ such that
$|\scr{C}|\leq\cont$ and for all $C\in\scr{C}$, $C$ is countably
tight and dense in $X$. Then $|X|\leq\cont$.
\end{corollary}

We compare Corollary~\ref{phcountable} with the following theorem of~\juhasz~and van Mill, which was also
extended to the power homogeneous case in~\cite{Car2018}. Observe they are variations of each other.

\begin{theorem}[\cite{JVM2018}, Theorem 4.1]
If a compactum $X$ is the union of countably many dense countably
tight subspaces and $X^\omega$ is homogeneous, then
$|X|\leq\cont$.
\end{theorem}

%%%%%%%%%%%%%%%%%%%%%%%%%% section %%%%%%%%%%%%%%%%%%%%%%%%%%%%
\section{Weak tightness and a second theorem of Balogh}
Our aim in this section is to give a variation
(Theorem~\ref{Balogh} below) of Theorem 2.2 in
Balogh~\cite{Bal2003} using the weak tightness $wt(X)$. It is
also related to Lemma 3.2 in~\cite{JVM2018} and results in an
alternative proof that $|X|\leq 2^{L(X)wt(X)\psi(X)}$ if $X$ is
Hausdorff, proved in~\cite{Car2018}.

\begin{theorem}\label{Balogh}
Let $X$ be a $T_1$ space. If $wt(X)\leq\kappa$, $X$ is
$\kappa^+$-compact, and $\psi(\closure{D},X)\leq 2^\kappa$ for
any $D\sse X$ satisfying $|D|\leq 2^\kappa$, then $d(X)\leq
2^\kappa$ and $X$ has at most $2^\kappa$-many $G_\kappa$-points.
\end{theorem}

\begin{proof}
For any $D\in [X]^{\leq\kappa}$ fix a family $\scr{U}_D$ of open
sets such that $\Meet\scr{U}_D=\closure{S(D)}$ and
$|\scr{U}_D|\leq 2^\kappa$. We will define by transfinite
induction a non-decreasing sequence of $\scr{C}$-saturated sets
$\{H_\alpha:\alpha<\kappa\}$ satisfying:

\begin{itemize}
\item[$1_\alpha$)] $|H_\alpha|\leq 2^\kappa$;
\item[$2_\alpha$)] if $X\minus\Un\scr{V}\neq\es$ for some
$\scr{V}\in[\Un\{\scr{U}_D:D\in[H_\alpha]^{\leq\kappa}\}]^{\leq\-
kappa}$, then $H_{\alpha+1}\minus\Un\scr{V}\neq\es$.
\end{itemize}

Put $H_0=S(x_0)$ for some $x_0\in X$ and let $\phi:\scr{P}\to X$
be a choice function extended by letting $\phi(\es)=x_0$. Assume
we have already defined the subsequence
$\{H_\beta:\beta<\alpha\}$. If $\alpha$ is a limit ordinal, then
put $H_\alpha=\Un\{H_\beta:\beta<\alpha\}$. If $\alpha=\gamma+1$,
then put
$H_\alpha=H_\gamma\un\Un\{S(\phi(X\minus\Un\scr{V})):\scr{V}\in[-
\Un\{\scr{U}_D:D\in[H_\gamma]^{\leq\kappa}\}]^{\leq\kappa}\}]\}$.
A counting argument ensures that $H_\alpha$ satisfies $1_\alpha$
and $2_\alpha$.

Now, put $H=\Un\{H_\alpha:\alpha<\kappa^+\}$. It is clear that
$|H|\leq 2^\kappa$. We show H is dense in $X$. Assume by
contradiction that there is some $p\in X\minus\closure{H}$. We
will show that this assumption implies the existence of an S-free
sequence $\{x_\alpha:\alpha<\kappa^+\}\sse H$. So, let
$\alpha<\kappa^+$ and suppose we have already chosen points
$\{x_\beta:\beta<\alpha\}\sse H$ and open sets
$V_\beta\in\scr{U}_{\{x_\xi:\xi<\beta\}}$ such that for every
$\beta<\alpha$ we have $\closure{\Un\{S(x_\xi):\xi<\beta\}}\sse
V_\beta\sse X\minus\{p\}$ and $\{x_\xi:\beta\leq\xi<\alpha\}\meet
V_\beta=\es$.  If $D_\alpha=\{x_\beta:\beta<\alpha\}$, then what
we are assuming implies $p\notin\closure{S(D_\alpha)}$. We may
then pick $V_\alpha\in\scr{U}_{D_\alpha}$ such that $p\notin
V_\alpha$. Furthermore, there exists an ordinal
$\gamma(\alpha)<\kappa^+$ such that $D_\alpha\sse
H_{\gamma(\alpha)}$ and so
$\{V_\beta:\beta\leq\alpha\}\sse\Un\{\scr{U}_D:D\in[H_{\gamma(-
\alpha)}]^{\leq\kappa}\}$. Since
$X\minus\Un\{V_\beta:\beta\leq\alpha\}\neq\es$, condition
$2_{\gamma(\alpha)}$ guarantees that we may pick a point
$x_\alpha\in
H_{\gamma(\alpha)+1}\minus\Un\{V_\beta:\beta\leq\alpha\}$.
Therefore, the construction can be carried out for all
$\alpha<\kappa^+$, producing an S-free sequence of length
$\kappa^+$.

Now we show every $G_\kappa$ point is contained in $H$. Assume by
contradiction that there exists a $G_\kappa$ point $p\in X\minus
H$ and fix a family of open sets $\{W_\alpha:\alpha<\kappa\}$
such that $\Meet\{W_\alpha:\alpha<\kappa^+\}=\{p\}$. Then let
$\scr{V}=\{U\in\Un\{\scr{U}_D:D\in[H]^{\leq\kappa}\}:p\notin
V\}$. Observe that $H$ cannot be covered by $\leq\kappa$ elements
of $\scr{V}$ because such a family would be contained in
$\Un\{\scr{U}_D:D\in[H_\alpha]^{\leq\kappa}\}$ for some
$\alpha<\kappa^+$ and this in turn would contradict condition
$2_\alpha$. So, there is an ordinal $\gamma<\kappa$ such that
$H\minus W_\gamma$ cannot be covered by $\leq\kappa$ elements of
$\scr{V}$. Now, by mimicking the argument used in the paragraph
above, we may again establish the existence of an S-free sequence
$\{x_\alpha:\alpha<\kappa^+\}\sse H\minus W_\gamma$.

\end{proof}

The following is well-known. We include its proof for
completeness.

\begin{proposition}\label{prop}
If $F$ is a closed subspace of a $T_1$ space $X$, then $\psi(F,
X)\leq |F|^{L(X)\psi(X)}$.
\end{proposition}

\begin{proof}
Let $\lambda=|F|$ and $\kappa=L(X)\psi(X)$. For all $x\in F$
there exists an open family $\scr{U}_x$ such that
$\{x\}=\Meet\scr{U}_x$ and $|\scr{U}_x|\leq\kappa$. Let
$\scr{U}=\Un\{\scr{U}_x:x\in F\}$ and
$\scr{F}=[\scr{U}]^{\leq\kappa}$. Note
$|\scr{U}|\leq\lambda\cdot\kappa$ and $|\scr{F}|\leq
(\lambda\cdot\kappa)^\kappa=\lambda^{\kappa}$.

Fix $y\in X\minus F$. For all $x\in F$ there exists
$U_x\in\scr{U}_x$ such that $y\notin U_x$. Then $\{U_x:x\in F\}$
is an open cover of $F$. As $L(X)\leq\kappa$ there exists
$\scr{V}_y\in\left[\{\scr{U}_x:x\in F\}\right]^{\leq\kappa}$ such
that $F\sse\Un\scr{V}_y$. Note $y\notin\Un\scr{V}_y$ and
$\scr{V}_y\in F$.

As $F\sse\Meet\left\{\Un\scr{V}_y:y\in X\minus F\right\}$ and
$y\notin\Meet\left\{\Un\scr{V}_y:y\in X\minus F\right\}$ for all
$y\in X\minus F$, it follows that
$F=\Meet\left\{\Un\scr{V}_y:y\in X\minus F\right\}$. Now,
$\{\scr{V}_y:y\in X\minus F\}\sse\scr{F}$ and therefore
$|\{\Un\scr{V}_y:y\in X\minus
F\}|\leq|\scr{F}\leq\lambda^\kappa$. This completes the proof.
\end{proof}

\begin{corollary}[\cite{Car2018}, Theorem 2.8]
If $X$ is a Hausdorff space then $|X|\leq 2^{L(X)wt(X)\psi(X)}$.
\end{corollary}

\begin{proof}
Let $L(X)wt(X)\psi(X)=\kappa$. $L(X)\leq\kappa$ implies that $X$
is $\kappa^+$-compact. By Theorem 2.5 in \cite{Car2018}, $|D|\leq
2^\kappa$ implies $|\closure{D}|\leq |D|^{wt(X)\psi_c(X)}\leq
|D|^{L(X)wt(X)\psi(X)}\leq 2^\kappa$. As $L(X)\psi(X)\leq\kappa$
and $|\closure{D}|\leq 2^{\kappa}$, by Proposition~\ref{prop} it
follows that $\psi(\closure{D},X)\leq 2^\kappa$. Now use either
conclusion of Theorem~\ref{Balogh}.
\end{proof}

%%%%%%%%%%%%%%%%%%%%%%%%%%%%%%%%%%%%%%%%%%%%%%%%%%%%%%%%%%%
\section{Weak tightness and calibers}
A cardinal $\kappa$ is a \emph{caliber} of a space X if every
family of open sets of cardinality $\kappa$ has a subfamily of
cardinality $\kappa$ with a non-empty intersection. At the end of
a survey paper on the Souslin number \cite{Sap1994}, B.E.
\sapirovskii~stated without proof the following:

\begin{proposition}[\cite{Sap1994}, Theorem 5.23]
Let $X$ be a \lindelof,~regular sequential space. If an
uncountable cardinal $\lambda\leq\cont^+$ is a caliber of $X$
then $|X|\leq\cont$.
\end{proposition}

Later, A.V. \arhangelskii~\cite{Arh2000} published a detailed
proof of this result in the space case $\lambda=\aleph_1$. His
proof is based on the following:

\begin{proposition}[\cite{Arh2000}, Theorem 5.1]
Let $X$ be a \lindelof,~regular space. If $t(X)=\aleph_0$ and
$\aleph_1$ is a caliber of $X$, then $d(X)\leq\cont$.
\end{proposition}

Our purpose here is to show that the previous result continues to
hold by replacing tightness with weak tightness. To achieve it,
we will modify the argument in the proof of Theorem 1 of
\cite{Bella2005}. Recall that a space X is said to be
quasiregular if for any non-empty open set $U$ there exists a
non-empty open set $V$ such that $\closure{V}\sse U$.

\begin{theorem}\label{caliber}
Let $X$ be a quasiregular space and let $\kappa=L(X)wt(X)$. If a
cardinal $\lambda$ satisfying $\kappa<\lambda\leq
\left(2^{\kappa}\right)^+$ is a caliber of $X$, then $d(X)\leq
2^{\kappa}$.
\end{theorem}

\begin{proof}
Assume by way of contradiction that $d(X)>2^\kappa$. Fix a choice
function $\eta:\scr{P}(X)\minus\{\es\}\to X$. We will define by
induction an increasing family $\{A_\alpha:\alpha<\lambda\}$ of
$\scr{C}$-saturated subsets of $X$ of cardinality not exceeding
$2^\kappa$ and a family $\{U_\alpha:\alpha<\lambda\}$ of
non-empty open subsets of $X$ in a such a way that for all
$\alpha<\lambda$,
\begin{enumerate}
\item $\closure{A_\alpha}\meet\closure{U_\alpha}=\es$, and
\item if $\scr{V}\sse\{U_\beta:\beta<\alpha\}$ satisfies
$|\scr{V}|\leq\kappa$ and $\Meet\scr{V}\neq\es$, then
$\eta(\Meet\scr{V})\in A_\alpha$.
\end{enumerate}
To justify the above construction, let us assume to have already
defined the $\scr{C}$-saturated sets $\{A_\beta:\beta<\alpha\}$
and the open sets $\{U_\beta:\beta<\alpha\}$. Since
$\alpha<\lambda$ and $\lambda\leq\left(2^{\kappa}\right)^+$, we
have $|\{U_\beta:\beta<\alpha\}|\leq |\alpha|\leq 2^\kappa$.
Consequently, the set
$B=\{\eta(\Meet\scr{V}):\scr{V}\sse\{U_\beta:\beta<\alpha\},|\scr{V}|\leq\kappa\textup{ and }\Meet\scr{V}\neq\es\}$ has
cardinality not exceeding $2^\kappa$. Let
$A_\alpha=\Un\{S(x):x\in B\}\un\Un\{A_\beta:\beta<\alpha\}$ and
note that $|A_\alpha|\leq 2^\kappa$. As we are assuming that
$d(X)>2^\kappa$, we may find a non-empty open set $U_\alpha$ such
that $\closure{A_\alpha}\meet\closure{U_\alpha}=\es$.

Since $\lambda$ is a caliber of $X$, there exists a set
$E\sse\lambda$ such that $|E|=\lambda$ and
$\Meet\{U_\alpha:\alpha\in E\}\neq\es$. We may fix an increasing
mapping $f:\lambda\to E$. Observe now that we are assuming
$\kappa^+\leq\lambda$. For any $\alpha<\kappa^+$ let
$x_\alpha=\eta(\Meet\{U_{f(\xi)}: \xi\leq\alpha\})$. If
$\beta<\alpha<\kappa^+$, then $x_\beta\in A_{f(\beta)+1}\sse
A_{f(\alpha)}$ and $x_\alpha\in U_{f(\alpha)}$ and so
$x_\beta\neq x_\alpha$. Consequently, the set
$Y=\{x_\alpha:\alpha<\kappa^+\}$ has cardinality $\kappa^+$ and
therefore, by $L(X)\leq\kappa$, there exists a complete
accumulation point $p$ for $Y$.

As pointed out in Lemma 2.7 of \cite{Car2018},
$\closure{\Un\{A_{f(\alpha)}:\alpha<\kappa^+\}}=\Un\{\closure{A_-
{f(\alpha)}:\alpha<\kappa^+\}}$. As
$Y\sse\Un\{A_{f(\alpha)}:\alpha<\kappa^+\}$, there exists some
$\gamma<\kappa^+$ such that $p\in\closure{A_{f(\gamma)}}\sse
X\minus\closure{U_{f(\gamma)}}$. Moreover, for each
$\gamma\leq\beta<\kappa^+$ the set $U_{f(\gamma)}$ occurs in the
definition of $x_\beta$ and consequently $x_\beta\in
U_{f(\gamma)}$. This means that the set
$W=X\minus\closure{U_{f(\gamma)}}$ is a neighborhood of $p$ such
that $|W\meet Y|\leq\kappa$, in contrast to the choice of $p$.
This completes the proof.
\end{proof}

For a space $X$ and $A\sse X$, the $\theta$-\emph{closure} of $A$
is defined as $cl_\theta(A)=\{x\in X:\closure{U}\meet
A\neq\es\textup{ whenever }U\textup{ is an open set containing
}x\}$. A set $D\sse X$ is $\theta$-\emph{dense} if
$X=cl_\theta(D)$ and the $\theta$-\emph{density} of $X$ is is
defined as $d_\theta(X)=\min\{|D|:D\textup{ is }\theta\textup{
dense in }X\}$. Dropping the assumption of quasiregularity, the
same proof above implies a similar result concerning
$d_\theta(X)$.

\begin{theorem}\label{calibertheta}
Let $X$ be a space and let $\kappa=L(X)wt(X)$. If a cardinal
$\lambda$ satisfying $\kappa<\lambda\leq
\left(2^{\kappa}\right)^+$ is a caliber of $X$, then
$d_\theta(X)\leq 2^{\kappa}$.
\end{theorem}

For a space $X$, the \emph{linear Lindel\"of degree} $lL(X)$
\cite{bella17} is
the least infinite cardinal $\kappa$ such that every increasing
open cover of $X$ has a subcover of size at most $\kappa$.
 $X$ is
\emph{linearly Lindel\"of} if $lL(X)$ is countable. It is clear
that $lL(X)\leq L(X)$. It is well-known that if $lL(X)\leq\kappa$
for a cardinal $\kappa$ then every set of cardinality $\kappa^+$
has a complete accumulation point. Examining the proofs of
Theorems~\ref{caliber} and \ref{calibertheta}, we see that both
still hold if $L(X)$ is subsituted with $lL(X)$.

\end{document}